\documentclass[runningheads,a4paper,oribibl]{llncs}

\usepackage{amssymb}
\usepackage{graphicx}

\usepackage{amsmath}
\usepackage{bbding}
\usepackage{bbm}
\usepackage{bm}
\usepackage{color}
\usepackage{algorithmic}
\usepackage{algorithm} 
\usepackage{listings}
\usepackage{subfigure} 
\usepackage{natbib}

\usepackage{url}
\urldef{\mailsa}\path|yzw137@psu.edu|
\urldef{\mailsb}\path|ludmil@psu.edu|

\newcommand{\keywords}[1]{\par\addvspace\baselineskip
\noindent\keywordname\enspace\ignorespaces#1}

\begin{document}

\mainmatter  

\title{Fourier Method for Approximating Eigenvalues of Indefinite Stekloff Operator\thanks{This work was supported in part by National Science Foundation
Grants DMS-1522615 and DMS-1720114}}

\titlerunning{FFT for Stekloff Eigenvalues Problem}

%
%
\author{Yangqingxiang Wu
\and  Ludmil Zikatanov 
 \Envelope
}


%
\authorrunning{Y. Wu, L. Zikatanov}

\institute{
Department of Mathematics, The Pennsylvania State University\\
University Park, PA 16802, USA\\
\mailsb\\
\url{http://personal.psu.edu/ltz1/}}

%
%

\toctitle{Fourier Method for Approximating Eigenvalues of Indefinite Stekloff Operator}
\tocauthor{Y.~Wu, L.~T.~Zikatanov}
\maketitle

\begin{abstract} We introduce an efficient method for computing the
  Stekloff eigenvalues associated with the Helmholtz equation. In
  general, this eigenvalue problem requires solving the Helmholtz
  equation with Dirichlet and/or Neumann boundary condition repeatedly. We
  propose solving the related constant coefficient Helmholtz equation
  with Fast Fourier Transform (FFT) based on carefully designed
  extensions and restrictions of the equation.  The proposed Fourier
  method, combined with proper eigensolver, results in an efficient
  and clear approach for computing the Stekloff eigenvalues.

\end{abstract}

\section{Introduction}

We consider the problem of computing the Stekloff eigenvalues
corresponding to the indefinite Helmholtz equation. The efficient
computation of such eigenvalues is needed in several numerical models.
For example, in inverse scattering, as discussed in
\cite{2016CakoniF_ColtonD_MengS_MonkP-aa}, these eigenvalues carry
information of the refractive index of an obstacle.  We introduce the
following boundary value problem: For a fixed $\eta$, find
$\lambda\in \mathbb{C}$ such that there is a non-trivial solution
$w\in H^1(\Omega)$ of the equation
\begin{equation}
\label{eq:SP}
\mathcal{L}(\alpha,\lambda;\eta):=\begin{cases}
-\Delta w-\eta^2 w =0\quad \mbox{in} \quad \Omega \\
\alpha\frac{\partial w}{\partial n}+\lambda w=0 \quad \mbox{on} \quad \Gamma=\partial \Omega.
\end{cases}
\end{equation}
We call $\lambda$ a Stekloff eigenvalue when $\alpha=1$ and
\eqref{eq:SP} has non-trivial solution. 

As pointed out in~\cite{2016CakoniF_ColtonD_MengS_MonkP-aa}, the
efficient computation of Stekloff eigenvalues is a challenging
task. In addition, another important application of the techniques
that we propose here is of interest in computing transmission
eigenvalues where the aim is to find the kernel of the difference of
two indefinite Stekloff operators~\cite{2017CakoniF_KressR-aa}.  The
efficient solution of such problems, whether direct or inverse,
requires fast solution of the Dirichlet problem, corresponding to the
operator $\mathcal{L}(0,1;k(x))$, and the Neumann problem,
corresponding to the operator $\mathcal{L}(1,0;k(x))$
in~\eqref{eq:SP}. Here, $k(x)$ is the wave number and in case of non
homogenous problem, such as $\mathcal{L}(\cdot,\cdot;\cdot) = (f,0)^t$, 
the data $f(x)$ is the external force.

The difficulties associated with solving the indefinite Helmholtz
equation numerically, especially in high frequency regimes, are well
known (see~\cite{brandt1997wave, ernst2012difficult}). Traditional
iterative methods, such as Krylov subspace methods or standard
MultiGrid (MG) and Domain Decomposition (DD) methods, are
inefficient. We refer to \cite{brandt1997wave, ernst2012difficult} for
discussion of such issues and guidance on how to tackle them.

Over the last two decades, different preconditioners and solvers for
the Helmholtz equation have been proposed. We refer to the classical
works by Brandt and Livshits~\cite{brandt1997wave}, and Elman et
al.~\cite{elman2001multigrid} for MG solvers and also to the more
recent developments in Helmholtz preconditioning presented
in~\cite{gander2015applying,osei2010preconditioning}.  More recent,
Enquist and Ying \cite{engquist2011sweeping} introduced the so called
sweeping preconditioners which were further extended by Eslaminia and
Guddati~\cite{eslaminia2016double} to double-sweeping preconditioners.
Stolk~\cite{stolk2013rapidly} proposed a DD preconditioner based on
special transmission conditions between subdomains. Other DD methods
are found in~\cite{chen2013source,zepeda2016method}.

In our focus are the computations of Stekloff eigenvalues and the
techniques which we propose here lead to efficient algorithms in many
cases of practical interest and provide preconditioners for the
Helmholtz problem. More importantly, our techniques easily extend to
the Maxwell's system because they are based on the Fourier method.

The rest of the paper is organized as follows.  We introduce the Fourier
method for solving the constant coefficients
boundary value problem in Section \ref{sc:FFThelmholtz} (Dirichlet)
and in Section \ref{sc:FFThelmholtz_Neumann} (Neumann). Further, in
Section \ref{sc:FFT_eig}, we formulate the Stekloff eigenvalue
problem and show how the FFT based Helmholtz solver can be applied. We
conclude with several numerical tests on Stekloff eigenvalue
computations as well as solution of the Helmholtz equation with
variable wave number.

\section{Periodic Extensions and Fourier Method}

\subsection{Dirichlet Boundary Conditions: 1d Case}\label{sc:FFThelmholtz}
To explain the ideas we consider the 1D version of~\eqref{eq:SP}
in the interval $(0,1)$:
\begin{equation*}\label{helmholtz1d}
-u'' - k^2 u = f, \quad u(0)=u(1) = 0.
\end{equation*}

After a discretization, using central finite difference, we
obtain the following linear system 
\begin{equation}\label{helmholtz-fd}
A^D\bm{u} = \bm{f}, \quad A^D=T^D-k^2 h^2 I\in \mathbb{R}^{n\times n},
\end{equation}
where $T^D=\operatorname{diag}(-1,2,-1)$ is a tri-diagonal matrix, $\bm u=(u_1,\ldots,u_n)^t$,
$u_0=u_{n+1} = 0$, $h=1/(n+1)$, $\bm f=h^2(f_1,\ldots,f_n)^t$, and
$u_j\approx u(jh)$, $f_j\approx f(jh)$, $j=1,\ldots n$.

Let us now consider the
same equation on a larger domain $(0,2)$  and
with periodic boundary conditions: 
\begin{equation}
-v''-k^2v = g, \quad v(0)=v(2),\quad v'(0)=v'(2).
\end{equation}
Analogous discretization approach leads to a
linear system for $\bm{v} = (v_1,\ldots,v_N)^t$, $N=2n+2$, which is as follows:
\begin{equation}\label{helmholtz-fd-p}
A^P\bm{v} = \bm{g}, \quad 
A^P = T^P- k^2h^2 I\in \mathbb{R}^{N\times N}.
\end{equation}
Here, $ e_1=(1,0,\ldots,0)^t$ and $e_N=(0,\ldots,0,1)^t$ are the
standard Euclidean basis vectors and
$T^P=\operatorname{diag}(-1,2,-1) - e_1 e_N^t- e_N e_1^t$ is a circulant matrix.
The right hand side $\bm{g} = h^2(g_1,\ldots, g_N)^t$ is a given vector in
$\mathbb{R}^{N}$ depending on $\bm f$, which we specify later.  The unknowns in this case are
$v_j\approx v(2j/N)$, $j=1,\ldots,N$. Notice that from the periodic
boundary conditions, we have $v_N\approx v(0)=v(2)$ and
$v_1\approx v(2/N) $ and this is reflected in the first and the last
equation in the linear system~\eqref{helmholtz-fd-p}.

The solution of systems with circulant matrices can be done
efficiently using the Fast version (FFT) of the Discrete Fourier
Transform (DFT) (see~\cite{1965CooleyJ_TukeyJ-aa} for a description of
FFT). The DFT is represented by an operator
$\mathcal{F}: \mathbb{C}^N\rightarrow \mathbb{C}^N$ represented by a
matrix (denoted again with $\mathcal{F}$) defined as:
$ \mathcal{F}_{jm} = \omega^{(j-1)(m-1)}, \quad
\omega=e^{-\frac{2i\pi}{N}}, and j=1,\ldots,N, \quad m=1,\ldots, N$.
As is well known, we have the DFT inversion formula:
\[
\mathcal{F}^{-1} = 
\frac{1}{N}\mathcal{F}^* = \frac{1}{N}\overline{\mathcal{F}}.
\]
Since $A^P$ here is a circulant matrix it is diagonalized by
$\mathcal{F}$~\cite{1965CooleyJ_TukeyJ-aa,golub2012matrix} and
\begin{equation}\label{diag1d}
\mathcal{F} A^P\mathcal{F}^{-1} = D^P = \operatorname{diag}(d_l),\quad
d_l = 4\sin^2\frac{\pi(l-1)}{N}-k^2h^2, \quad l=1,\ldots N. 
\end{equation}
As a consequence of this proposition, the solution $\bm{v}$ to the
problem~\eqref{helmholtz-fd-p} can be obtained  by
\begin{equation}\label{solvep}
 \bm v = \mathcal{F}^{-1} (D^P)^{-1}\mathcal{F} \bm g.
\end{equation}

Let us now consider the special case when $\bm g$ in
\eqref{helmholtz-fd-p} corresponds to an ``odd'' function. We have the
following simple result.
\begin{proposition}\label{1dodd}
  If $N=2n+2$ and $\bm g $ satisfies $g_{j} = -g_{N-j}$,
 for $j=n+2, n+3,\ldots, 2n+1$ and $g_{n+1}=g_{2n+2}=0$, then the solution
  to~\eqref{helmholtz-fd-p} satisfies the relation:
\begin{equation}\label{e:vvv}
v_{j} = -v_{N-j}, \quad j=n+2,\ldots,2n+1.
\end{equation}
\end{proposition}
\begin{proof} 
  We note that by assumption, $\bm g$ is an ``odd'' function with
  respect to the middle of the interval $(0,2)$. Since $A^D$ is
  an invertible matrix, let $\bm u$ satisfy 
  $[A^D\bm u]_j=\bm g_j$, $j=1,\ldots,n$.  Next, we define
  $\bm v=E\bm u\in \mathbb{C}^{N}$ where $E$ is the extension operator
  defined below in \eqref{extension} and it is immediate to verify
  that $A^P\bm v=\bm g$.  Since $A^P$ is also invertible, $\bm v$,
  which satisfies this relation, is the unique solution of
  $A^P\bm v=\bm g$. From the definition of $E$, we conclude that
  $\bm v$ also satisfies~\eqref{e:vvv}.
\end{proof}
Based on this observation, to solve the Dirichlet problem,
we can define a linear operator $B$ (which we will soon prove
equals $(A^D)^{-1}$) as follows: 
Let the extension $\bm g:=E\bm{f}$ be defined as
\begin{equation}\label{extension}
\mathbb{C}^N \ni \bm g = E\bm f,\quad 
g_j = \begin{cases}
f_{j},\quad j=1,\ldots,n,\\ 
0,\quad j=n+1,\\
-f_{N-j}, \quad j=n+2,\ldots,2n+1\\
0,\quad j=N.
\end{cases}
\end{equation}
Here, $N=2n+2$ and we also have the following restriction operator
\begin{equation}
\label{restriction}
\mathbb{C}^n \ni \bm w = R \bm v, \quad 
w_j =v_{j+1}, \quad j=1,\ldots n, \quad \bm v\in\mathbb{R}^N.
\end{equation}
We then set
\begin{equation}\label{B-def}
B\bm f = R \mathcal{F}^{-1} (D^P)^{-1} \mathcal{F} E \bm f.
\end{equation}
As the next proposition shows, $B$ provides the exact solution to
problem~\eqref{helmholtz-fd}.
\begin{proposition}\label{inverse1d}
With $B$ defined in~\eqref{B-def} we have 
\[
B = (A^D)^{-1},\quad\mbox{or equivalently,}\quad \bm u = B \bm f. 
\]
\end{proposition}
\begin{proof}
  We notice that $R=(I_{n}|\bm 0)\in \mathbb{R}^{n\times N}$, 
  and $E=(I_{n}, \bm 0_{n\times 1},-\tilde{I}_{n},\bm 0_{n\times 1}
  )^T\in \mathbb{R}^{N\times n}$, where $I_n$ is the $n\times n$ identity matrix and $\tilde{I}_{n}=(\delta_{i,n+1-j})_{ij}$.   Computing the product $A^DR(A^P)^{-1}E$ then shows that:
  \begin{equation}\label{asterisk}
    A^DR(A^P)^{-1}E=A^D(I|0)(A^P)^{-1}E=(A|0)(A^P)^{-1}E=(I|0)E=I.
  \end{equation}
  Indeed, the identities in~\eqref{asterisk} are verified by direct calculation:
  \begin{equation}
    \begin{split}
      &((A|0)A_p^{-1})_{ij}\\
      &=-\frac{1}{N}(\mathcal F^{-1} D_p ^{-1}\mathcal F)_{i-1,j}+\frac{2}{N}(\mathcal F^{-1} D_p^{-1} \mathcal F)_{i,j}-\frac{1}{N}(\mathcal F^{-1} D_p^{-1} \mathcal F)_{i+1,j}\\
      &=\frac{1}{N}\sum_{l=1}^N(-\bar{\omega}^{(i-2)(l-1)}+(2-k^2h^2)\bar{\omega}^{(i-1)(l-1)}-\bar{\omega}^{i(l-1)})d_l\omega^{(k-1)(j-1)}\\
	&=\frac{1}{N}\sum_{l=1}^N(2-k^2h^2-2\cos(\frac{2\pi (l-1)}{N}))d_l\omega^{(k-1)(j-i)}\\
	&=\frac{1}{N}\sum_{l=1}^N\omega^{(k-1)(j-i)}=\delta_{ij}.\\
	\end{split}
      \end{equation}
      This completes the proof. 
\end{proof}

\subsubsection{Generalization to Higher Dimensions (Dirichlet Problem)}
\label{sc:Dir_high_d}
We now consider the Helmholtz operator $\mathcal{L}(0,1;k(x))$, defined
in~\eqref{eq:SP} in $d$-dimensions, i.e. we take $\Omega=(0,1)^d$.
Discretization with standard central finite differences, we have  the
linear system:
\begin{equation}\label{eq:system2d_dir}
A^D\bm{u} = \bm{f},\quad  A^D = \sum_{j=1}^d  \left(I^{\otimes(j-1)} \otimes T^D \otimes I^{\otimes(d-j)}\right)
  -k^2h^2 I^{\otimes d} \in {\mathbb R}^{n^d\times n^d}.
\end{equation}
Here $M^{\otimes p}: = \underbrace{M \otimes \ldots \otimes M}_{p\quad copies}$ for any matrix $M$ and $T^D$ as in \eqref{helmholtz-fd} is the triangular matrix.

The extension and restriction operators in higher dimensions can be written as $E_d:=E^{\otimes d}$ and $R_d:=R^{\otimes d}$.
$\bm g=E_d \bm f$ is the ``odd'' extension of $\bm f$. As in the 1D case, the linear system for the extended Helmholtz equation is:
\begin{equation}\label{eq:system2d_perio}
A^P\bm{v} = \bm{g},\quad  A^P = \sum_{j=1}^d  \left(I^{\otimes(j-1)} \otimes T^P \otimes I^{\otimes(d-j)}\right)
  -k^2h^2 I^{\otimes d} \in {\mathbb R}^{N^d\times N^d},
\end{equation}
where $T^P$ has been defined in \eqref{helmholtz-fd-p}. As in the one
dimensional case~\eqref{diag1d}, the matrix $A^P$ is diagonalized by
the multidimensional DFT $\mathcal{F}_d=\mathcal{F}^{\otimes d}$.
The multidimensional version of~\eqref{diag1d} then is:
\[
\mathcal{F}_d A^P\mathcal{F}_d^{-1} = D^P_d:=\sum_{j=1}^d  \left(I^{\otimes(j-1)} \otimes D^P \otimes I^{\otimes(d-j)}\right)
  -k^2h^2 I^{\otimes d},
\]
As a consequence, we obtain inversion formula similar to the one
presented in~Proposition~\ref{inverse1d}. To show such representation,
we need the following result.
\begin{lemma}\label{ADAP} Let $E$ and $R$ be extension and restriction operator defined in \eqref{extension} and \eqref{restriction} respectively.  Then following identity holds
\begin{equation}
A^D = R^{\otimes d}  A^P  E^{\otimes d}.
\end{equation}
\end{lemma}
\begin{proof}
By using the standard properties of the tensor product, we have
\begin{eqnarray*}
R^{\otimes d} A^P  E^{\otimes d} &=& R^{\otimes d} \left(\sum_{j=1}^d \left(I^{\otimes(j-1)} \otimes T^P \otimes I^{\otimes(d-j)}\right)
-k^2h^2  I^{\otimes d}\right) E^{\otimes d} \\
& = & R^{\otimes d}\left(\sum_{j=1}^d  \left(E^{\otimes(j-1)} \otimes T^P E \otimes E^{\otimes(d-j)}\right)
-k^2h^2 E^{\otimes d}\right) \\
& = & \sum_{j=1}^d  \left(I^{\otimes(j-1)} \otimes RT^P E \otimes I_{n-1}^{\otimes(d-j)}\right)
-k^2h^2  I^{\otimes d}.
\end{eqnarray*}
It is straightforward to check that $RT^PE=T^D$.
Thus, $ R^{\otimes d}  A^P  E^{\otimes d}=A^D$.
\end{proof}
The following theorem gives the representation of the inverse of the discretized Dirichlet problem in the multidimensional case.
\begin{theorem}\label{inverseanyd}
The inverse of $A^D$ can be written as
\begin{equation}
\label{B-def-highdim}
(A^D)^{-1}=R_d (A^P)^{-1} E_d,\quad\mbox{or equivalently,}\quad \bm u = R_d (A^P)^{-1} E_d \bm f.
\end{equation}
\end{theorem}
\begin{proof}
Clearly, it is sufficient to prove that $A^DR_d (A^P)^{-1} E_d=I$.  
\[
A^DR_d (A^P)^{-1} E_d=R_dA^PE_dR_d(A^P)^{-1}E_d=R^{\otimes d}A^PE^{\otimes d}R^{\otimes d}(A^P)^{-1}E^{\otimes d}=I,
\]
by the using properties of the matrix tensor product, Lemma \ref{ADAP}
and identity $RE=I$.
\end{proof}

Notice here, all the above results can be extended to
rectangle (non-square) domain without too much difference. As a
conclusion, we can solve the constant coefficient Helmholtz equation
with Dirichlet boundary condition by FFT with complexity $O(n\lg n)$,
where $n$ is the problem size.

\subsection{Neumann Boundary Conditions} \label{sc:FFThelmholtz_Neumann}

We next discuss the Fourier method for solving the Neumann problem. We
begin by laying out the details for 1D case.

\subsubsection{Neumann Problem in 1D}
Let's consider the following 1D Helmholtz equation with one side
Neumann boundary condition in $(0,1)$:
\begin{equation}\label{helmholtz-1DNeumann}
-u'' - k^2 u = f, \quad u(0)=0, \quad u'(0)=g.
\end{equation}
Finite difference discretization gives the linear system of equations
$A\bm{u}= \bm{F}$ where 
\begin{equation}\label{helmholtz-1DNeumann-fd}
  A=\operatorname{diag}(-1,2-k^2h^2,-1)-(1-\frac{1}{2}k^2h^2)e_{n+1}e_{n+1}^t\in \mathbb{R}^{(n+1)\times (n+1)},
\end{equation}
and also $\bm u=(u_1,\ldots,u_{n+1})^t$, $h=1/(n+1)$,
$u_j\approx u(jh)$, $f_j\approx f(jh)$, for $j=1,\ldots n+1$, and
$\bm F=(h^2f_1,\ldots,h^2f_n,\frac{1}{2}h^2f_{n+1}+hg)^t$.

With Fourier method for the Dirichlet problem in mind, we do ``even"
extension of the system \eqref{helmholtz-1DNeumann-fd} to get a Toeplitz
system similar to \eqref{helmholtz-fd}:
\begin{equation}\label{helmholtz-1DNeumann-even-fd}
A^e\bm{u}^e = \bm{F}^e, \quad A^e=\operatorname{diag}(-1,2-k^2 h^2,-1) \in \mathbb{R}^{M\times M},
\end{equation}
where $M=2n+1$, $\bm u^e=(u_1,\ldots,u_{M})^t$, and
$\bm F^e=(h^2f_1,\ldots,
h^2f_n,h^2f_{n+1}+2hg,h^2f_{n},\ldots,h^2f_1)^t$. By symmetry, the
solution of system \eqref{helmholtz-1DNeumann-even-fd}, when
restricted on interval $(0,1)$, will be the same as solution of system
\eqref{helmholtz-1DNeumann-fd}. Even though the problem size has been
doubled, the extended system is Toeplitz, thus can be solved by
Fourier method from Section \ref{sc:FFThelmholtz}.
 
In summary, we have
the following inverse of the Neumann operator, that is the solution of :
\[
A^{-1}\bm{F} =  R\mathcal{F}^{-1} (D^p)^{-1} \mathcal{F}E_{o}E_{e}\bm{F}
\]
The operators involved in the definition above are (from right to
left): even extension, odd extension followed by the inverse of the
periodic problem and then restriction. More precisely, for the even extension $E_e$ we have, 
\begin{equation}
\mathbb{C}^M \ni \quad 
 [E_e\bm{F}]_j = \begin{cases}
F_{j},\quad j=1,\ldots,n,\\ 
2F_{n+1}, \quad j=n+1, \\
F_{M+1-j}, \quad j=n+2,\ldots,2n+1.\\

\end{cases}
\end{equation}
Next, for the extension as to odd functions/vectors we have:
\begin{equation}
\mathbb{C}^P \ni \bm [E_oE_e\bm{F}]_j = \begin{cases}
F_{j}^e,\quad j=1,\ldots,M,\\ 
0,\quad j=M+1,\\
-F_{N-j}^e, \quad j=M+2,\ldots,2M+1\\
0,\quad j=2M+2.
\end{cases}
\end{equation}
Finally, we have the diagonal $D^p=\operatorname{diag}(4\sin (j\pi/M))_{j=0}^{M-1}$ and the restriction $R$:
\[
\bm v\in \mathbb{R}^P,\quad  \mathbb{C}^{n+1} \ni \bm u = R \bm v, \quad u_j =v_{j}^p, \quad j=1,\ldots n+1.
\]

\subsubsection{Generalization to Higher Dimensions (Neumann problem)}
The Fourier method for Neumann problem can also be generalized to
any dimension $d$ in a fashion similar to the procedure given earlier
for the Dirichlet
problem. Just for illustration, if we consider
$d=2$ and  $\Omega=(0,1)^2$ with the following boundary conditions:
\begin{equation}
\label{eq:anyDNeumann}
\frac{\partial u}{\partial	x}=g, \quad \text{on} \quad \Gamma_1=\{x=1\}\times(0,1)  \\
\end{equation}
and homogeneous Dirichlet conditions elsewhere.
To solve the resulting linear system
we first do an ``even'' extension of the data and we arrive at:
\begin{equation}\label{helmholtz-2DNeumann-even-fd}
A^e\bm{u}^e = \bm{F}^e, \quad A^e=I_M\otimes T^d_n+T^d_M\otimes I_n-k^2 h^2I_{Mn} \in \mathbb{R}^{Mn\times Mn},
\end{equation}
where   $T^d_j=\operatorname{tridiag}(-1,2,-1) \in \mathbb{R}^{j\times j}$,  $M=2n+1$, $\bm u^e=(\bm u_1,\ldots,\bm u_{M})^t$, and
$\bm F^e=(h^2\bm f_1,\ldots, h^2\bm f_n,h^2\bm f_{n+1}+2hg,h^2\bm
f_{n},\ldots,h^2\bm f_1)^t$.  Clearly, on $\Omega$, the restriction of
$\bm u^e$ is the same as the solution $\bm u$ of the Neumann problem.
Now, for the solution of \eqref{helmholtz-2DNeumann-even-fd} we can
apply the method we have already described in \S\ref{sc:Dir_high_d}.

\section{Stekloff Eigenvalue Computation with Fourier Method}
\label{sc:FFT_eig}
\subsection{Variational Formulation} Multiplying the first equation in \eqref{eq:SP} by $v\in H^1(\Omega)$ and 
  integrating by parts, we get:
\[\int_\Omega \nabla w\nabla v-\eta^2\int_\Omega wv=-\lambda \int_\Gamma wv.\]

 Define $A(\eta):H^1(\Omega)\rightarrow H^{-1}(\Omega)$ as
 \begin{equation}
\langle A(\eta)w,v\rangle:=(\nabla w, \nabla v)_\Omega-\eta^2(w,v)_\Omega
\quad \forall v\in H^1(\Omega),
\end{equation}
where  $(\cdot, \cdot)_\Omega$ is the $L^2$ inner product and $\langle \cdot, \cdot\rangle$ is duality pairing between $H^{-1}(\Omega)$ and $H^1(\Omega)$. The Stekloff 
operator, or Dirichlet-to-Neumann(DtN), $S(\eta)$ can be defined in two steps:

 Firstly, for any $f\in H^{1/2}(\Gamma)$, define $f_0\in H_0^1(\Omega)$ as the unique function satisfies:
  $$\langle A(\eta)f_0,v_0\rangle=-\langle A(\eta)(Ef),v_0\rangle, \quad \forall v_0\in H_0^1(\Omega),$$
 where $Ef$ is $H^1$-bounded extension of $f$, e.g. harmonic extension.
 
  Secondly, define the action of $S(\eta):H^{1/2}(\Gamma)\rightarrow H^{-1/2}(\Gamma)$ as
  \begin{equation}
  \label{eq:S_def}
  \langle S(\eta)f,g\rangle_{1/2}=\langle A(\eta)(f_0+Ef),Eg\rangle,\quad 
  \forall g\in H^{1/2}(\Gamma),
  \end{equation}
  where $\langle\cdot,\cdot\rangle_{1/2}$ is the duality pairing
  between $H^{1/2}(\Gamma)$ and $H^{-1/2}(\Gamma)$.

\begin{lemma} 
\label{lem: SP}
The equation \eqref{eq:SP} has a non-trivial solution 
$\Leftrightarrow$ $\lambda$ is an eigenvalue of the Stekloff  operator $S(\eta)$.
\end{lemma}
\begin{proof}

We prove one side of this equivalence  here since the other half is similar. Suppose $u$ is solution to \eqref{eq:SP}, then
\[
 \langle A(\eta)u, v_0 \rangle=0,\quad \forall v_0 \in H^1_0(\Omega).
\]

 From the above equation, we have 
\[
\langle A(\eta)(u-Eu_\Gamma),v_0\rangle = -\langle A(\eta)Eu_\Gamma,v_0 \rangle \quad \forall v_0\in H^1_0(\Omega),
\] where $u_\Gamma :=u|_\Gamma$ as the trace of $u$ on $\Gamma$. Denote $u_0:=u-Eu_\Gamma$. It's easy to see that $u_0\in H_0^1(\Omega)$ and the following equation holds:
\[
\langle S(\eta)u_\Gamma, g\rangle_{1/2}=\langle A(\eta)(u_0+Eu_\Gamma), Eg\rangle =(\lambda u_\Gamma, g).
\]
This shows that $\lambda$ must be the eigenvalue of $S(\eta)$.
\end{proof}

This lemma implies that solving problem \eqref{eq:SP} is equivalent to
solving the eigenvalue problem for $S(\eta)$, for given $\eta$. In the
next section we describe an efficient method for this task, namely the
Fourier method for Stekloff eigenvalues.

\subsection{Neumann-to-Dirichlet and Dirichlet-to-Neumann Operators} 
Let $\mu\in L^2(\Gamma)$ and define $w_\mu$ to be the solution of
equation \eqref{eq:SP} with Neumann boundary condition, i.e. $w_{\mu}$
satisfies
 \begin{equation}
 \label{eq:Neu_variational}
 \langle A(\eta)w_\mu, v \rangle=\langle \mu, v_\Gamma \rangle_{1/2},\quad \text{for any} \quad v\in H^1(\Omega).
  \end{equation} 
Taking the trace of $w_\mu$, we can define Neumann-to-Dirichlet (NtD) operator $T$ as $T\mu=w_\mu|_{\Gamma}$. After discretization, we
  have the following linear system:
\begin{equation}
\begin{bmatrix}
A_{II}  & A_{IB}\\
A_{BI}  & A_{BB}\\
\end{bmatrix}\begin{bmatrix}
\bm w_I \\
\bm w_B
\end{bmatrix}=\begin{bmatrix}
\bm 0\\
\bm \mu
\end{bmatrix},
\end{equation}
where $\bm w_I$ and $\bm w_B$ are solution of
\eqref{eq:Neu_variational} restricted in $\Omega$ and on $\Gamma$
respectively. Then the discretized NtD operator $T_h$ corresponding to
the NtD operator $T$ is
\begin{equation}
\label{eq:defTh}
T_h\bm \mu=\begin{bmatrix}
\bm 0, I
\end{bmatrix}
\begin{bmatrix}
A_{II}  & A_{IB}\\
A_{BI}  & A_{BB}\\
\end{bmatrix}^{-1}\begin{bmatrix}
\bm 0\\
\bm \mu
\end{bmatrix}.
\end{equation}
To discretize the Stekloff operator $S$ in analogous fashion, we also
consider the Dirichlet boundary problem of \eqref{eq:SP} with boundary
data $f\in H^{1/2}(\Omega)$.  With the same discretization and
the same ordering as for the Neumann problem, we have the following linear
system:
\begin{equation}
\begin{bmatrix}
A_{II}  & A_{IB}\\
\bm 0  & I\\
\end{bmatrix}\begin{bmatrix}
\bm w_I \\
\bm w_B
\end{bmatrix}=\begin{bmatrix}
\bm 0\\
\bm f
\end{bmatrix}.
\end{equation}  
Clearly, $\bm w_b = \bm f$ as expected and
the discrete version of \eqref{eq:S_def} gives the action of $S_h$ as
\begin{equation}\label{eq:defSh}
S_h\bm f=(A_{BB}-A_{BI}A_{II}^{-1}A_{IB})\bm f.
\end{equation}
We have the following Lemma, which shows that $T_h$ is the inverse of $S_h$.
\begin{lemma}\label{ADAPx} For the operators $T_h$ and $S_h$ defined
  in~\eqref{eq:defTh} and \eqref{eq:defSh}, respectively, we have
  \(S_hT_h=I\).
\end{lemma}
\begin{proof}
  We just need to show $S_hT_h\bm \mu= \bm \mu$ for any $\bm\mu$. This
  can be easily proved by Block-LU factorization.
\begin{equation}
\begin{split}
S_hT_h\bm \mu
&=S_h\begin{bmatrix}
\bm 0,\bm I
\end{bmatrix}
\begin{bmatrix}
A_{II}  & A_{IB}\\
A_{BI}  & A_{BB}\\
\end{bmatrix}^{-1}\begin{bmatrix}
\bm 0\\
\bm \mu
\end{bmatrix}\\
&=S_h\begin{bmatrix}
\bm 0,\bm I
\end{bmatrix}
\bigg(\begin{bmatrix}
I  & \bm 0\\
A_{BI}A_{II}^{-1}  & I\\
\end{bmatrix}\begin{bmatrix}
A_{II}  & A_{IB}\\
\bm 0  & S_h\\
\end{bmatrix}\bigg)^{-1}\begin{bmatrix}
\bm 0\\
\bm \mu
\end{bmatrix} =\bm \mu.
\end{split}
\end{equation}
\end{proof}

The results from the previous two sections show that we can
efficiently compute the Stekloff eigenvalues of small magnitude as
well as large magnitude. For example, the eigenvalues of small
magnitude of $S_h$ can be approximated by the reciprocal of the
eigenvalues of NtD operator $T_h$. The action of $T_h$, which is
needed repeatedly in such a procedure, can be efficiently computed by
the Fourier method applied to the solution of the Helmholtz-Neumann
problem as we have discussed earlier. For the eigenvalues of largest
magnitude the same applies, except that we need the action of $S_h$,
which requires fast solution of the corresponding Dirichlet problems,
which we also described earlier.

\section{Numerical Examples}
\label{sc:num_exp}
\subsection{Helmholtz Equation with Varying Wave Number}
Until now, we have shown that Fourier method can solve the constant
coefficient Helmholtz equation with Dirichlet boundary condition
exactly with nearly linear complexity. Using this exact solver as
preconditioner, we can also solve certain type of varying coefficient
problem, e.g. slowly varying coefficient.

In this numerical example, we consider a homogeneous Dirichlet problem
in domain $\Omega=(0,1)^2$ with uniform external force $f=1$ and the
following velocity fields:
\renewcommand{\labelenumi}{\textbf{\arabic{enumi}.~}}
\begin{enumerate}
\item $c(x_1,x_2)=\frac{4}{3}[1-0.5\exp(-0.5(x_1-0.5)^2]$, 
\item  $c(x_1,x_2)=\frac{4}{3}[1-0.5\exp(-0.5(x_1-0.5)^2+(x_2-0.5)^2)]$.
\end{enumerate}

We study how the preconditioner behaves when $\omega$ and $n$ vary in
the same ratio. In this way, the percentages of positive and negative
eigenvalues are fixed, which poses the most challenge for computation.
We record the numbers of iterations for GMRES method to converge
(error tolerance $10^{-6}$) and total times used for computing in
Table \ref{tb:vary_omg_n}. For $w/2\pi=3.2$, the wave number fields
$k_i(x,y)$ and resulting wave field $u_i(x,y)$ with $i=1,2$ have been
shown in Figure \ref{fig:k_and_u}.

\begin{table}[ht]
  \caption{Varying $\omega$ and $n$, number of GMRES iterations and time to converge}
  \centering
  \begin{tabular}{ | c |  c | c | c | c | c |}\hline
$\frac{\omega}{2\pi}$ & $n^2$& $N_1$ & $T_1$  & $N_2$  & $T_2$  \\ \hline
 0.8& $50^2$ & 4& 0.068 & 4 & 0.062 \\  \hline
    1.6&  $100^2$ & 5 & 0.074& 5 & 0.076 \\  \hline
    3.2&   $200^2$& 6 & 0.12 &  6&  0.12 \\   \hline
    6.4&  $400^2$ & 13 & 0.60& 10 & 0.46 \\  \hline
  \end{tabular}
    \label{tb:vary_omg_n}
\hspace{0.7cm}
\end{table}

\begin{figure}[!ht]
\centering
\subfigure[wave number function $k_1(x,y)$]{\includegraphics[width=5cm]{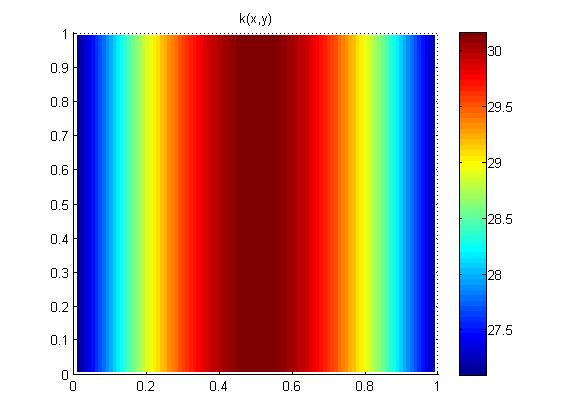}}
\subfigure[wave simulated $u_1(x,y)$]
{\includegraphics[width=5cm]{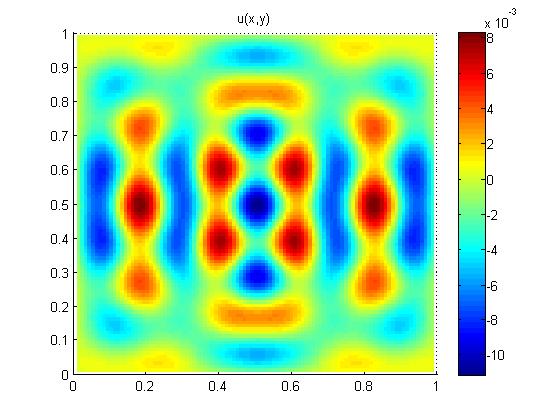}}
\hfill
\subfigure[wave number function $k_2(x,y)$]{\includegraphics[width=5cm]{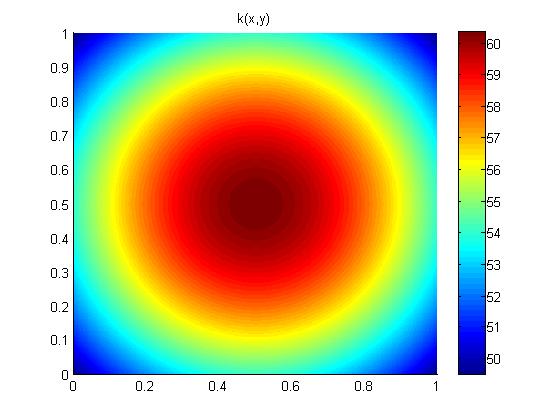}}
\subfigure[wave simulated $u_2(x,y)$]
{\includegraphics[width=5cm]{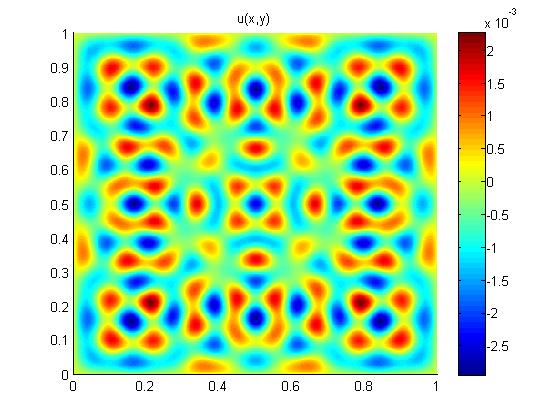}}
\caption{Simulation results for different velocity fields. }
\label{fig:k_and_u}
\end{figure}

\subsection{Computing Stekloff Eigenvalues}
In this numerical example, we apply our Fourier-based method, which has been introduced in Section \ref{sc:FFT_eig}, to the Stekloff eigenvalues problem. 
We create a function handle that solves the Neumann boundary problem with FFT Helmholtz solver proposed, then use MATLAB \verb$eigs$ function to calculate the eigenvalues. 
For different $\eta$, the smallest six eigenvalues in magnitude can be found in Table \ref{tb:eigenvalues}.

\begin{table}[ht]
  \caption{Smallest six Stekloff eigenvalues in magnitude for different $\eta$}
  \centering
  \begin{tabular}{ | c |  c | c | c | c | c | c |}\hline
$\eta$ & $\lambda_1$& $\lambda_2$ & $\lambda_3$  & $\lambda_4$ & $\lambda_5$   &$\lambda_6$ \\ \hline
 0.5&    0.2132 &   0.2158&   0.5561 & 0.8910 &  0.8910 &  2.6103\\  \hline
    1&   0.2162 & 0.2190 &  0.5832 &  1.0003  &  1.0003 &    4.1828\\  \hline
    2&  0.2302  &  0.2327 &  0.7548 &   2.4831 &   2.4831  &  -2.6194\\   \hline
    4&   0.4909 &  0.4909 & 0.6099&   0.7888 &      1.5170 &     1.5170\\  \hline
  \end{tabular}

    \label{tb:eigenvalues}
\end{table}

\section{Conclusions}
In this paper, we proposed an efficient method for finding eigenvalues
of indefinite Stekloff operators.  The main tool that we developed is
a fast Fourier method for solving constant coefficient Helmholtz
equation with Dirichlet or Neumann boundary condition on rectangular
domain.  The resulting algorithm is efficient, transparent, and easy
to implement. Our numerical experiments show that such algorithm works
also as a solver for the Helmholtz problem with mildly varying
coefficient (non-constant wave number).  Another pool of important
applications will be the computation of transmission eigenvalues,
where our method has the potential to provide an efficient
computational tool.

\bibliographystyle{plainnat}
\bibliography{mybib}

\end{document}